\documentclass[11pt]{article}

\usepackage{amssymb,amsmath}

\newtheorem{theo}{\bf \eh Theorem}[section]

\newtheorem{prop}[theo]{\bf \eh Proposition}
\newtheorem{lema}[theo]{\bf \eh Lemma}
\newtheorem{claim}[theo]{\bf \eh Claim}
\newtheorem{corol}[theo]{\bf \eh Corollary}
\newtheorem{conjecture}[theo]{Conjecture}

\def\R{\mathord{\mathbb R}}
\def\Q{\mathord{\mathbb Q}}
\def\Z{\mathord{\mathbb Z}}
\def\N{\mathord{\mathbb N}}

\def\notmid{\mathrel{\vert\hspace*{-.4em}/}}

\def\eh{\hspace*{\parindent}}

\def\qed{\hfill$\Box$}

\title{Counting Spectral Radii of Matrices with Positive Entries}
\author{J.\ A.\ Dias da Silva and Pedro J.\ Freitas\\
Centro de Estruturas Lineares e Combinat\'oria\\
Departamento de Matem\'atica \\
da Faculdade de Ci\^encias da Universidade de Lisboa\\
\texttt{japsilva@fc.ul.pt\ pedro@ptmat.fc.ul.pt}}
\date{April 1, 2012}

\begin{document}

\maketitle

\begin{abstract}
The sum-product conjecture of Erd\H os and Szemer\'edi states that, given a finite set $A$ of positive numbers, one can find asymptotic lower bounds for $\max\{|A+A|,|A\cdot A|\}$ of the order of $|A|^{1+\delta}$ for every $\delta <1$. In this paper we consider the set of all spectral radii of $n\times n$ matrices with entries in $A$, and find lower bounds for the cardinality of this set. In the case $n=2$, this cardinality is necessarily larger than $\max\{|A+A|,|A\cdot A|\}$. 
\end{abstract}

\section{A conjecture inspired by Erd\H os and Szemer\'edi}

In additive combinatorics, some classical results are estimates of cardinality of the ranges of some classical functions such as the sum or the product, when the domain is resticted to a certain finite set $A\times A$, as a function of $|A|$, the cardinality of $A$.\medskip

Among the best known results in this area, in finite characteristic, are the Cauchy-Davenport theorem and the Erd\H os-Heilbronn conjecture (see \cite{Na} and \cite{SH} for instance), where bounds are established for the cardinality of sumsets in $\Z_p$.\medskip

In characteristic zero, one of such results is the conjecture of Erd\H os and Szemer\'edi that no $A\subseteq \mathbb{N}$ will yield values of $|A+A|$ and $|A\cdot A|$ that are both ``small''---by ``small'' we mean a polynomial of degree 1 in $|A|$. To be more precise, the conjecture of Erd\H os and Szemer\'edi is as follows. We write $a(x)\gg b(x)$ to mean that $a(x)$ is greater than $c\cdot b(x)$ for a certain constant $c$ and for all $x$. 

\begin{conjecture}
Let $A\subset \N$ be a finite set. Then for all $\delta<1$,
$$\max\{|A+A|,|A\cdot A|\}\gg |A|^{1+\delta}.$$
\end{conjecture}

The best proved result so far is $\delta \geq 1/3$, in \cite{So}.\medskip

The aim of this paper is to extend this type of results to functions which are somewhat close to these classical functions. One example (which we will not address in this paper) is the elementary symmetric functions (see for instance \cite{SG1} and \cite{SG2}). Another example is the spectral radius function, defined on the set of matrices with entries in a finite subset of the positive reals. We will call such matrices positive, and denote by $\R^+$ the set of positive reals, by $\R^+_0$ the set of  non-negative reals and, for a set $A$, we denote by $A^{m\times n}$ the set of all $m\times n$ matrices with entries in $A$.\medskip

Let $M\in (\R^+_0)^{n\times n}$, $M\neq 0$. By the Perron-Frobenius theorem, the spectral radius of $M$, which we will denote by $\rho(M)$, is a positive eigenvalue of $M$. For $A\subseteq \R^+$, a finite set, we define
$$\Omega_n(A)=\{\rho(M)\; :\; M\in  A^{n\times n}\},$$
and we denote by $nA$ the set of all sums of $n$ elements of $A$, and by $A^{\times n}$ the set of all products of $n$ elements of $A$ (the first notation taken from \cite{Na}).

\begin{prop}
\label{pr:sp}
We have $nA\subseteq \Omega_n(A)$ and $|A^{\times n}|\leq |\Omega_n(A\cup\{0\})|$. Therefore,
$$|\Omega_n(A)|\geq |nA|,$$
$$|\Omega_n(A\cup\{0\})|\geq\max\{|nA|,|A^{\times n}|\}.$$
For $n=2$, we have
$$|\Omega_2(A)|\geq\max\{|A+A|,|A\cdot A|\}.$$
\end{prop}
\begin{proof} Let $a_1,\ldots,a_n\in A$. Take the matrix of $A^{n\times n}$ in which all rows are equal to
$$\left[ a_1\ a_2\ \cdots \ a_n \right].$$
Then $a_1+\cdots+a_n$ is an eigenvalue of this matrix, associated with eigenvector $[1\ 1\ \ldots\ 1]^T$. Since the matrix has rank 1, all other eigenvalues are 0, thus $a_1+\cdots+a_n$ is the spectral radius, therefore $nA \subseteq \Omega _n(A) \subseteq \Omega_n(A\cup\{0\})$. \medskip

For the second result, consider the matrix

$$
M=\left[\begin{array}{ccccc}
0&a_1&0&\cdots&0\\
0&0&a_2&\cdots &0\\
\vdots&\vdots&\vdots&\ddots&\vdots\\
0&0&0&\cdots&a_{n-1}\\
a_n&0&0&\cdots&0
\end{array}\right]$$
The eigenvalues of this matrix are the complex $n$-th roots of $a_1 \cdots a_n$, and therefore this matrix has spectral radius $\sqrt[n]{a_1 \cdots a_n}$ (all eigenvalues have this absolute value). Therefore $\{\sqrt[n]{a_1 \cdots a_n}\; :\; a_1,\ldots, a_n\in A\}\subseteq \Omega_n(A\cup\{0\})$ and
$$|\Omega_n(A\cup\{0\})|\geq | \{\sqrt[n]{a_1 \cdots a_n} :\; a_1,\ldots, a_n\in A\}|=|A^{\times n}|.$$  

Thus we conclude that $|\Omega_n(A\cup\{0\})|\geq\max\{|nA|,|A^{\times n}|\}$.\medskip 

In case $n=2$, fix $a\in A$ and for $a_1,a_2\in A$, take the matrix
$$aI_2+\begin{bmatrix}
0 & a_1\\ a_2 & 0
\end{bmatrix}\in A^{2\times 2}$$
with spectral radius $a+\sqrt{a_1a_2}$. The map $a_1a_2\to a+\sqrt{a_1a_2}$, from $A\cdot A$ to $\Omega_2(A)$ is injective, so $|\Omega_2(A)|\geq |A\cdot A|$.\qed
\end{proof}\medskip

The previous result has a clear relationship to the problem of finding lower bounds for both $|A+A|$ and $|A\cdot A|$. The conjecture of Erd\H os and Szemeredi, along with Proposition \ref{pr:sp} leads naturally to the following conjecture. 
\begin{conjecture}
Let $A\subset \N$ be a finite set. Then for all $\delta<1$,
$$|\Omega_2(A)|\gg |A|^{1+\delta}.$$
\end{conjecture}

From the existing result we can already say that $|\Omega_2(A)|\gg |A|^{4/3}$. In the last section of the paper, we present a stronger conjecture. 

\section{Lower bounds for $|\Omega_n(A)|$ with set restrictions}

Let $A\subseteq \N$ be a finite set and $c\in \N$. Define $c*A:=\{ca: a\in A\}$ as in \cite{Na}, and 
$$\chi_A:= \max \left\{ \frac{|A\cap c*A|}{|A|} : c \text{ is not a perfect square} \right\}.$$

\begin{prop}Let $A \in \N$ be a finite set such that $\chi_A>0$. Then 
$$|\Omega_2(A)|\geq \chi_A|A|^2.$$
\end{prop}
\begin{proof}
Let $c\in \N$ be such that $\chi_A=|A\cap c*A|/|A|>0$.
Let $A\cap c*A=\{cb_1,\ldots,cb_t\}$, $b_1,\ldots, b_t\in A$, $t=\chi_A |A|$. For each pair $(a_i,b_j)\in A\times (A\cap c*A)$, consider the following matrix in $A^{2\times 2}$:
$$\begin{pmatrix}
a_i & b_j\\ cb_j & a_i
\end{pmatrix} =
\begin{pmatrix}
a_i & 0\\ 0 & a_i
\end{pmatrix}+
\begin{pmatrix}
0 & b_j\\ cb_j & 0
\end{pmatrix}.$$
By the additive decomposition above, we see that the spectral radius of this matrix is $a_i+\sqrt{c}b_j$. Since $c$ is not a perfect square, for $(a_i,b_j)\neq (a'_i,b'_j)$, we have $a_i+\sqrt{c}b_j\neq a'_i+\sqrt{c}b'_j$. Thus we obtain $|A|\times t$ different spectral radii, and thus $|\Omega_n(A)|\geq |A|\times t = \chi_A|A|^2$. Since $\chi_A|A|\geq 1$, we also have the other inequality.\qed
\end{proof}

\begin{corol}
Let $A$ be a geometric progression in which the ratio is not a perfect square. Then $|\Omega_2(A)|\geq (|A|-1)|A|$. 
\end{corol}
\begin{proof}
In this case, $\chi_A=(|A|-1)/|A|$.\qed
\end{proof}\medskip

We now consider the prime factors of the elements of $A$. It is clear that given a set $A$, one can divide all its elements by the greatest common divisor of its elements, obtaining a set we will denote by $A'$, with
$$|A|=|A'|\quad \text{ and }\quad |\Omega_n(A)|=|\Omega_n(A')|.$$

Let $\mathbb{P}$ be the set of prime numbers. We now define
$$\pi_A = \max \left\{\frac{|A'\cap p*\N|}{|A'|}\; : p\in \mathbb{P},\ (A'\cap p*\N)\setminus p^{2}*\mathbb{N} \ne \emptyset \right\},$$
or zero if $\{p\in \mathbb{P},\ (A'\cap p*\N)\setminus p^{2}*\mathbb{N} \ne \emptyset \}=\emptyset$.

We now consider sets $A$ for which $\pi_A > 0$ and with $|A|>1$. For these sets, we have that for some prime $p$, not all elements of $A'$ are multiples of $p$, and not all elements of $A'\cap p*\N$ are multiples of $p^2$. 

Before the result we present a technical lemma. 

\begin{lema}
Let $a_0,\ldots,a_{n-1},c,d\in \R$. Then the characteristic polynomial of 
$$
\left[\begin{array}{cccccc}
0&c& \\
&&c&\\
&&&\ddots& \\
&&&&c \\
&&&& & d\\
a_0 & a_1& a_2&\cdots&\cdots& a_{n-1}
\end{array}\right]$$
is 
$$d\sum_{i=0}^{n-2} \left( c^{n-i-2}(-1)^{n-i}a_ix^i \right) -a_{n-1}x^{n-1}+x^n,$$
or, written otherwise, 
$$(-1)^nc^{n-2}da_0+(-1)^{n-1}c^{n-3}da_1x+ \cdots + d a_{n-2}x^{n-2} - a_{n-1}x^{n-1}+x^n.$$
\end{lema} 
\begin{proof}
It can be seen using Laplace expansion on the last row. \qed
\end{proof}

\begin{theo}
For any finite set $A\subseteq \N$ with $\pi_A > 0$, 
$$|\Omega_n(A\cup \{0\})|\geq \pi_A |A|^{n-1} \geq |A|^{n-2}.$$
\end{theo}
\begin{proof} If $|A|=1$, the result is obvious. Now we consider that $|A|>1$. 
As before, $|\Omega_n(A\cup \{0\})|=|\Omega_n(A'\cup \{0\})|$.
Let $p\in \mathbb{P}$ such that $\pi_A=|A'\cap p*\N|/|A'|>0$. Fix 
$$c\in A'\setminus p\N\quad \text{and}\quad d\in(A'\cap p*\N)\setminus p^2*\N.$$ 
For each $(a_1,\ldots, a_{n-2}) \in (A')^{n-2}$ and $a_{n-1}\in A'\cap p*\N$, consider the matrix
$$
M(a_1,\ldots, a_{n-1})=\left[\begin{array}{cccccc}
0&c& \\
&&c&\\
&&&\ddots& \\
&&&&c \\
&&&& & d\\
c & a_1& a_2&\cdots&\cdots& a_{n-1}
\end{array}\right].$$
By the previous lemma, the characteristic polynomial of this matrix is
$$f(x)=(-1)^nc^{n-1}d+(-1)^{n-1}c^{n-3}da_1x+ \cdots + d a_{n-2}x^{n-2} - a_{n-1}x^{n-1}+x^{n}.$$
By Eisenstein's criterion, this polynomial is irreducible in $\mathbb{Q}[x]$. Therefore,
$f(x)$ is the minimal polynomial of the spectral radius of $M(a_1,\ldots, a_{n-1})$.

This means that for different elements in $(A')^{n-2}\times (A'\cap p*\N)$, we get different characteristic polynomials, all monic and all irreducible. This implies that they are the minimal polynomials of the spectral radii of the matrices associated to them. Therefore, these spectral radii have to be different, since their minimal polynomial over $\Q$ is not the same. In other words, the following map is injective.
$$\begin{matrix}
(A')^{n-2}\times (A'\cap p*\N) & \to & \Omega_n(A\cup \{0\})\\
(a_1,\ldots,a_{n-1}) & \mapsto & \rho(M(a_1,\ldots,a_n))
\end{matrix}$$

Therefore we get
$$|\Omega_n(A\cup \{0\})|\geq |A|^{n-2}\times |A'\cap p*\N| = \pi_A|A|^{n-1}\geq |A|^{n-2},$$
as we claimed.
\qed
\end{proof}

\begin{corol}
Let $A$ be a geometric progression with prime ratio. Then $|\Omega_n(A\cup\{0\})|\geq (|A|-1)|A|^{n-2}$. 
\end{corol}
\begin{proof}
If $A$ is a geometric progression with prime ratio $p$, then 
$$A'=\{a,pa,p^2a,\ldots,p^ta\},$$
with $p\notmid a$.
In this case, $\pi_A=(|A|-1)/|A|$.\qed
\end{proof}\medskip

\section{Lower bounds for $|\Omega_n(A)|$ for $A\subseteq \R^+$}

Now we recall a simple lower bound that we have for $|A+A|$.

\begin{prop}
Let $A\subseteq \N$ be a finite set. We have
$$|A+A|\geq 2|A|-1.$$
\end{prop}

Let $A=\{a_1<\cdots<a_k\}$. The technique used in proving this proposition consists in building an increasing sequence of $2(|A|-1)+1$ sums: you start with $a_1+a_1$ and increase successively one of the summands, and then the other. In each summand we have $|A|-1$ increases, which yields a total sequence of length $2(|A|-1)+1$. \medskip

We know from \cite[Th.\ 2.7]{Va} that for a positive matrix, each time you increase an entry, the spectral radius also increases. Therefore, using the same technique, one could get a lower bound of $n^2(|A|-1)+1$. However, with an extra result about the possibility of increasing the spectral radius by exchanging entries in a row, one can build a longer sequence of matrices with increasing spectral radii.\medskip

This is the result we are about to prove, using a generalization of this technique.

\begin{theo} Keeping the previous notation, we have the following lower bounds for $|\Omega_n(A)|$ if $n,|A|>1$ and $\max\{n,|A|\}>2$.\medskip

If $|A|<n$: 
$$n\frac{(|A|-1)|A|(2|A|-1)}{3} + n(|A|-1)(|A|-2)(n-|A|) + n|A|(n-|A|+1) -n+1.$$

If $|A|\geq n$: 
$$n\frac{(n-1)n(2n-1)}{3}+n^3(|A|-n)+n^2-n+1.$$ 
We note that the formulas coincide if $|A|=n$.
\end{theo}

We notice that if we fix $n$ and let $|A|$ increase, the lower bound we get is $O(|A|)$, but if we let both $|A|$ and $n$ increase, the bound gets better. For instance if $|A|=n$, the lower bound is $O(n^4)$. \medskip

We now prove this result in two steps: first we establish some auxiliary results, and then build a sequence of matrices, with strictly increasing spectral radius, with length given by the formulas above. 

\subsection*{Initial results}

\begin{prop}\label{rows_perron}
Let $M=[a_{ij}]\in (\R^+)^{n\times n}$ be a matrix in which the first $m$ rows have sum $a$ and the remaining rows are equal, with sum $b<a$. Let $(v_1,\ldots,v_n)$ be the Perron vector. Then for $r\leq m < s$ we have $v_r>v_s$. 
\end{prop}
\begin{proof} Let $\rho$ be the spectral radius of $M$, we know that $b<\rho<a$. Take $1\leq t\leq m$ such that $v_t=\min \{v_i : 1\leq i\leq m\}$. It is straightforward to check that $v_{m+1}=v_{m+2}=\ldots=v_n$. We then wish to show that $v_t>v_n$. Suppose, to the contrary, that $v_t\leq v_n$. This would imply that $v_t\leq v_j$ for all $1\leq j\leq n$ and  
$$\rho v_t = \sum_{j=1}^n a_{tj}v_j\geq \sum_{j=1}^n a_{tj}v_t = av_t,$$
which is false. Therefore we must have $v_t>v_n$ as we claimed.\qed
\end{proof}
\medskip

It is known that the entries of the Perron vector of a positive matrix are not always in the same order as the row sums---see for instance the matrix $A^2$ on p.\ 1159 of \cite{FHSSW} (the authors thank Hans Schneider for calling their attention to this paper). The result is true for $2\times 2$ positive matrices, by the previous result. One might think that it is also true for all $3\times 3$ positive matrices, but this is false: the requirement that two row sums are equal is necessary. The authors thank Shmuel Friedland for producing the next counter-example.\medskip

Take a $2\times 2$ matrix 
$$\begin{pmatrix} a & b\\ c & d \end{pmatrix}$$
with Perron vector $(v_1,v_2)$. Assume that $v_1>v_2$ and $a+b>b+c$. If we take $a_t:= a+t$, $b_t=b-tv_1/v_2$, we get that the matrix 
$$\begin{pmatrix} a_t & b_t\\ c & d \end{pmatrix}$$
is still positive for small positive values of $t$, has the same Perron vector but the first row has sum smaller that $a+b$. 

Taking one such value for $t$, consider the following $3\times 3$ matrix:
$$\begin{pmatrix} 
a/2 & a/2 & b\\ 
a_t/2 & a_t/2 & b_t\\ 
c/2 & c/2 & d 
\end{pmatrix}.$$
Its Perron vector is $(v_1,v_1,v_2)$. The first and second entries are equal, yet the corresponding row sums are not the same. \medskip

The proof of the next result is adapted from \cite{Sc}.

\begin{prop}
\label{schwarz}
Let $M=[m_{ij}]\in (\R^+)^{n\times n}$ and let $v=(v_1,\ldots,v_n)$ be its Perron vector. Suppose that for some indices $r,s,t$, with $s<t$, 
$$(m_{rt}-m_{rs})(v_s-v_t)>0.$$
Then the matrix obtained from $M$ by exchanging entries $m_{rs}$ and $m_{rt}$ has a greater spectral radius than $M$.
\end{prop}
\begin{proof} For $N\in (\R^+)^{n\times n}$, denote its spectral radius by $\rho(N)$. It is known that if we can find a positive vector $x$ and $\tau>0$ such that $Nx\geq \tau x$, $Nx\neq \tau x$ then $\rho(N)>\tau$ (see \cite[p.\ 28]{BP}). 

Let $L$ be the matrix obtained from $M$ exchanging entries $(r,s)$ and $(r,t)$. For $i\neq r$, we have that
$$\sum_{j=1}^n l_{ij}v_j= \sum_{j=1}^n m_{ij}v_j = \rho(M) v_i,$$
and in row $r$, 
\begin{eqnarray*}
\sum_{j=1}^n l_{rj}v_j-\rho(M)v_r & = & \sum_{j=1}^n (l_{rj}-m_{rj})v_j\\
& = & (m_{rs}-m_{rt})v_t +(m_{rt}-m_{rs})v_s \\
& = & (m_{rt}-m_{rs})(v_s-v_t).
\end{eqnarray*}
If the last number is positive, then $Lv\geq \rho(M)v$ and $Lv\neq \rho(M)v$, therefore $\rho(L)>\rho(M)$.\qed\end{proof}

\subsection*{An initial sequence of matrices}

We now build a sequence of matrices, as in the result mentioned in the beginning of this section.

Let
$$A=\{a_1<a_2<\cdots<a_k\},\quad |A|=k.$$

We consider the partial order defined componentwise in $\N^n$. From now on we will just consider non-decreasing elements in $\N^n$. We say that a certain element $(i_1,\ldots,i_n)\in \N^n$ has an increment in position $(r,r+1)$ if $i_r<i_{r+1}$. The value of this increment will be $i_{r+1}-i_r>0$.\medskip

Consider now the following sequence of $n$-tuples: start by the element that has all entries equal to 1, and use the following rule iteratively, until we reach the $n$-tuple with all entries equal to $k$.

\begin{itemize}
\item If there is an increment of 2 in position $(r,r+1)$, then increase the element in the $r$-th position by 1.
\item If there is no increment of 2, then look for the rightmost element that is not equal to $k$ and increase that element by 1.
\end{itemize}

As an example, suppose $n=5$ and $k= 4$. The sequence is
$$\begin{array}{l}
(1,1,1,1,1),\\
(1,1,1,1,2),\\
(1,1,1,1,3), (1,1,1,2,3),\\
(1,1,1,2,4), (1,1,1,3,4), (1,1,2,3,4),\\
(1,1,2,4,4), (1,1,3,4,4), (1,2,3,4,4),\\
(1,2,4,4,4),(1,3,4,4,4),(2,3,4,4,4)\\
(2,4,4,4,4),(3,4,4,4,4),\\
(4,4,4,4,4).
\end{array}$$

It is easy to check that in each $n$-tuple, the order of the entries is non-decreasing, with possible increments of 1 or 2. In addition, there can only be one increment of 2. We note that we get a total number of $1+n(k-1)$ rows, since we have $n$ entries, and each one is incremented $k-1$ times.\medskip

Taking $(r_1,\ldots,r_n)$ the $i$-th element on this list, we will call $L_i$ the row matrix $[a_{r_1} \ \cdots \ a_{r_n}]\in A^{1\times n}$. Now we build a sequence of matrices in $A^{n\times n}$ defining the rows of each matrix. The first matrix will have all rows equal to $L_1$. Then we replace, successively, from row 1 to row $n$, $L_1$ by $L_2$, until we get a matrix with all rows equal to $L_2$. Now we do the same with $L_3$, and proceed in this manner, until we get a matrix in which all entries are equal to $a_k$. \medskip

In this sequence of matrices, each element is obtained by the previous one by replacing an element $a_j$ in a some entry by $a_{j+1}$. Therefore each matrix entry starts at $a_1$ and 
ends at $a_k$, and thus the sequence must have length 
\begin{equation}
n^2(k-1)+1, 
\label{length}
\end{equation}
counting the $k-1$ transitions of each entry plus the initial matrix. \medskip

As we mentioned, we know from \cite[Th.\ 2.7]{Va} that if in a positive matrix the value of one entry is increased, then the spectral radius also increases. So in this sequence of matrices the spectral radius strictly increases as we move forward in the sequence.\medskip

Before we move on, we count how many rows there are with $t$ increments, for later use.\medskip

\begin{enumerate}
\item If $t<\min\{n-2, k-2\}$, then all increments are either at the beginning or the end of the row. In each case, we get $t+1$ rows: one with no increment of 2, and $t$ others with an increment of 2 (there are $t$ positions where this increment can appear). We thus get $2(t+1)$ rows.

\item Rows with $t=\min\{n-2,k-2\}$ increments. We consider two cases. 
\begin{enumerate}
\item $k < n$. In this case, we have, along with all rows with increments either at the beginning or at the end, there are also rows in which the $k-2$ increments are between positions $(2,3)$ and $(n-2,n-1)$---call these {\em internal increments}. Since the element before the first increment can happen in positions 2 through $n-k+1$, we have a total of $n-k$ possibilities for this.  We must also have an increment of 2 at some point, and this one can go in $k-2$ different positions. So the number of rows with internal increments is $(n-k)(k-2)$.
\item $k \geq n$. In this case, the computation of the number of rows with $n-2$ increments is done as in 1, thus getting $2(n-1)$ rows.
\end{enumerate}

\item Maximum number of increments.
\begin{enumerate}
\item If $k<n$, there are $n-k+1$ rows with $k-1$ increments, as this is the number of positions for the element before the first increment.
\item If $k\geq n$, to count the rows with $n-1$ increments (the maximum possible number), we consider two possibilities. If the element in position 1 is between 1 and $k-n$ (in case $k>n$), we get $n$ rows, since either we have no increment of 2 or we do have one, and there are $n-1$ positions where this increment can be. If the element in position 1 is $k-n+1$, we get just one row. The number of rows is thus $n(k-n)+1$.
\end{enumerate}

\end{enumerate}

To confirm these reckonings, we are going to count all the matrices we get, grouping them by number of increments. We know this number has to be $n^2 (k-1)+1$, as in equation (\ref{length}). 

Notice that each row $L_i\neq L_1$ appears in $n$ matrices of the following form: for $1\leq r\leq n$, consider the matrix in which all rows from 1 to $r$ are equal to $L_i$ and the remaining ones are equal to $L_{i-1}$. \medskip

\begin{enumerate}
\item Case $k<n$. 
$$\underbrace{2\frac{k(k-1)}{2}}_{\text{increments at the ends}} + \underbrace{(n-k)(k-2)}_{\text{internal increments}}+n-k+1=$$
$$=n(k-1)+1.$$
Subtracting 1, for row $L_1$, multiplying the remaining ones by $n$, and adding 1 again, we get $n^2 (k-1)+1$.
\item Case $k\geq n$. 
$$\underbrace{2\frac{n(n-1)}{2}}_{\text{increments at the ends}} +\, n(k-n)+1=$$
$$=n(k-1)+1.$$
\end{enumerate}
The final computation is the same. 

\subsection*{Modifying the sequence}

We will now alter this sequence, adding some more matrices. Let $s_i$ be the row sum of $L_i$, we have that if $i<j$, then $s_i<s_j$. 
In view of what will follow, we consider matrices $M$ such that:

\begin{itemize}
\item The first $r$ rows have sum $s_i$.
\item The remaining ones are equal to $L_{i-1}$.
\end{itemize} 
By Proposition \ref{rows_perron}, the $r$ first entries of the Perron vector of this matrix are bigger than the remaining ones. Let $v=(v_1,\ldots,v_n)$ be the Perron vector.
\smallskip

Let $m_{ri}$ and $m_{rj}$, $i<j$, be two elements of row $r$ of $M$. By what we have seen in Proposition \ref{schwarz}, if $m_{ri}<m_{rj}$ and $v_i>v_j$, if we exchange $m_{ri}$ and $m_{rj}$, we get a new matrix with strictly larger spectral radius. Now we establish how many times we can do this in the row matrix $L_i$.

\begin{claim} If the row matrix $L_i$ is placed in position $r$ of matrix $M$ and has $t$ increments, then we can make $t$ of these exchanges, thus obtaining $t+1$ matrices with distinct spectral radii (including the one before any exchange is done). 
\end{claim}
To verify this, we use the following algorithm, which we separate in three cases for the sake of clarity.\medskip

{\bf Case 1}. Every increment in $L_i$ appears before $(r,r+1)$ (which means that from position $r+1$ onward, all entries of $L_i$ are equal). 

Let $m_{rr'}$ be the entry before the last increment, $r'\leq r$, we place elements  $m_{rr'}$, $m_{rr'-1}$, $m_{rr'-2}$, \ldots successively in position $n$ of this row. For instance, if $r=4$, e $L_i=[3\ 4\ 5\ 6\ 6\ 6]$, the sequence is
$$[3 4 5 6 6 6]\to [3 4 6 6 6 5] \to [3 5 6 6 6 4] \to [4 5 6 6 6 3].$$
It is easy to check that these exchanges are in the conditions of Proposition \ref{schwarz}.\smallskip

{\bf Case 2}. Every increment of $L_i$ appears after position $(r,r+1)$ (so, before position $r$, all entries of $L_i$ are equal). We do a similar algorithm, only here we do the exchanges with the first entry in the row, starting with the entry after the first increment. We present an example, similar to the previous one:
$$[3 3 3 4 5 6]\to [433356] \to [533346] \to [633345].$$
Again, these exchanges are in the conditions of Proposition \ref{schwarz}.\smallskip

{\bf Case 3}. There is an increment in position $(r,r+1)$, assume $1<r<n-1$. We first apply the algorithm described in case 1 for entries $r,r-1,\ldots$, as long as the entries decrease, and afterwards we apply the algorithm of case 2 to entries $r+1,\ldots$, as long as the entries increase. We do not touch entry in position $n$, which has been changed. Here is an example with $r=3$, e $L_i=[1\ 1\ 2\ |\ 3\ 4\ 6\ 6]$ (we added a vertical line to mark $(r, r+1)$).
\begin{eqnarray*}
&& [112|3466] \to [116|3462] \to [126|3461]\\
&& [126|3461] \to [326|1461] \to [426|1361] \to [626|1341]
\end{eqnarray*}
(We repeated the row $[126|3461]$ to separate both types of exchanges.)
Notice that even if we had exchanged entry in position 1, it would still be smaller that entries in positions $r+1,\ldots,n-1$. 

In this case, as we can see, we can always have $t+1$ exchanges, but for simplicity, we will not perform the last one, in order to have just $t$ exchanges (obtaining $t+1$ rows) as in the other cases.
\medskip

In case $r=1$, we skip the first part, and just do the algorithm as in case 2. For instance, if $L_i=[2\ |\ 3\ 4\ 6\ 6]$, we would have
$$[2|3466] \to [3|2466] \to [4|2366] \to [6|2346].$$
In this case, we can only do $t$ exchanges. This also happens if some other cases: $r=n-1$ or $t=r-1$ or if there is an increment on position $(n-1,n)$.\medskip

Therefore, we can always do $t$ exchanges, obtaining $t+1$ matrices as before, all with different spectral radii. This proves the claim.\qed\medskip

We now describe the procedure for obtaining the final sequence of matrices. 
\begin{enumerate}
\item We start with a matrix with all rows equal to $L_i$---at the very beginning, $i=1$. This matrix has spectral radius equal to the row sum of $L_i$, $s_i$, since $(1,\ldots,1)$ is an eigenvector.
\item Look for the first row equal to $L_i$, replace it by $L_{i+1}$ and in this row do all possible exchanges, as described in the previous algorithm.
\item We repeat step 2 until row $L_i$ in position $n-1$ is replaced and all exchanges are done. After this we do all exchanges in row $L_i$ in position $n$. We now have a matrix with spectral radius strictly smaller than $s_{i+1}$.
\item We now replace all rows by $L_{i+1}$, obtaining a matrix with spectral radius $s_{i+1}$, and return to step 1, unless row $L_{i+1}$ is the last row in the sequence.
\end{enumerate}

It can be easily checked that all the matrices in this sequence have the form described in Proposition \ref{rows_perron}, and so we have the desired inequalities in the entries of the Perron vector.\medskip

We now calculate the final number of matrices in the sequence. We claim that, for $L_i\neq L_1$ each of the row matrices $L_i'$ obtained from $L_i$ by exchanging its entries provides $n$ matrices. Taking $1\leq r\leq n$, the argument is similar to the one presented before, with a little adjustment for $r=n$. For a row matrix $L_i$, let $L_i^{j\prime}$ be the last row obtained from $L_i$, after all exchanges are done, assuming $L_i$ occupies the row $j$ (we remark that the exchanges to be made depend on $j$).
\begin{itemize}
\item If $r\neq n$, consider the matrix in which rows $1\leq j\leq r-1$ are equal to $L_i^{j\prime}$, row $r$ is equal to $L_i'$ and the remaining rows are equal to $L_{i-1}$.
\item If $r=n$:
\begin{itemize}
\item If $L_i'\neq L_i$, consider the matrix in which all rows in positions $1 \leq j\leq n-1$ are equal to $L_{i+1}^{j\prime}$ and row $n$ is equal to $L_i'$;
\item If $L_i'=L_i$, consider the matrix in which all the rows are equal to $L_i$. 
\end{itemize}
\end{itemize}

Therefore, each row with $t$ increments provides $n(t+1)$ matrices to the sequence. \medskip

Before the final computation, we recall a formula that will be useful. For $m>2$, 
$$\sum_{t=1}^{m-2} (t+1)^2 = \sum_{t=2}^{m-1} t^2 = \frac{(m-1)m(2m-1)}{6}-1.$$

We must pay special attention to the two rows with 0 increments, the first and the last ones of the sequence. This first one will provide only one matrix, whereas the other will provide $n$, as all others. This accounts for the first two summands in the next expressions.\medskip

Case $k<n$. 
\begin{eqnarray*}
&& 1+n+2n\displaystyle\sum_{t=1}^{k-2} (t+1)^2 + n(k-1)(n-k) + nk(n-k+1) =\\
&& \quad = n\frac{(k-1)k(2k-1)}{3} + n(k-1)(k-2)(n-k) + \\
&& \quad \quad + nk(n-k+1) -n+1.
\end{eqnarray*}

Case $k\geq n$. Similarly, we get
\begin{eqnarray*}
&& 1+n+2n\displaystyle\sum_{t=1}^{n-2} (t+1)^2 + n^2(n(k-n)+1)\\
&& \quad = n\frac{(n-1)n(2n-1)}{3}+n^3(k-n)+n^2-n+1.
\end{eqnarray*}

This finishes the proof.\qed

\section{Final remarks}

The authors made some attempts in the way of proving that $|\Omega_n(A)|\leq |\Omega_{n+1}(A)|$. To do this one might try to establish that, for some $v_1,v_2\in A^{n\times 1}$ and $a\in A$, if $\rho(M)<\rho(M')$, then
$$\rho \left(\begin{bmatrix} M &  v_1\\ v_2^T & a\end{bmatrix}\right)
< 
\rho \left(\begin{bmatrix} M' &  v_1\\ v_2^T & a \end{bmatrix}\right).$$
It may be possible to find some bordering for which this is true, but it is not true in general. For instance,
$$\rho \left(\begin{bmatrix} 1 & 2\\ 1 & 1 \end{bmatrix}\right) = 2.414 < 
\rho \left(\begin{bmatrix} 2 & 1\\ 1 & 1 \end{bmatrix}\right) = 2.618,$$
and yet
$$\rho \left(\begin{bmatrix} 1 & 2 & 1\\ 1 & 1 & 3\\ 1 & 1 & 3 \end{bmatrix}\right) = 4.791 > 
\rho \left(\begin{bmatrix} 2 & 1 & 1\\ 1 & 1 & 3\\ 1 & 1 & 3\end{bmatrix}\right) = 4.732.$$

We could have deduced the inequalities about the spectral radii just using the theory of exchanges that we have presented.\medskip

Some computational experiments have suggested the following conjectures. Define $r_k:=\min \{|\Omega_2(A)|: |A|=k\}$.

\begin{conjecture}
We have
$$r_k=\frac{1}{2}(2k^3-k^2+k)= kT(k)+(k-1)T(k-1),$$
where $T(k)=k(k+1)/2$ is the $k$-th triangular number. This is sequence {\tt A081436} in the On-Line Encyclopedia of Integer Sequences (http://oeis.org/ visited in April 2011). Moreover, 
$r_k=|\Omega_2(A)|$ if $A$ is a geometric progression with $|A|=k$.
\end{conjecture}

\begin{conjecture} Let $B$ be a set of positive reals with $|B|=k-1$. Then $|\Omega_2(B\cup\{0\}|)| < r_k$. 
\end{conjecture}

\end{document}